\documentclass[12pt]{amsart}

\usepackage{amsmath, amscd, amssymb, latexsym}
\usepackage{graphicx}
\usepackage{hyperref}
\usepackage{cleveref}
\usepackage{mathptmx}
\usepackage{esint}

\theoremstyle{plain}

\newtheorem{theorem}{Theorem}[section]
\newtheorem{lemma}[theorem]{Lemma}
\newtheorem{proposition}[theorem]{Proposition}
\newtheorem{corollary}[theorem]{Corollary}

\numberwithin{equation}{section}
\allowdisplaybreaks[3]

\crefname{theorem}{Theorem}{Theorems}
\crefname{lemma}{Lemma}{Lemmas}
\crefname{proposition}{Proposition}{Propositions}
\crefname{corollary}{Corollary}{Corollaries}
\crefname{question}{Question}{Questions}
\crefname{conjecture}{Conjecture}{Conjectures}
\crefname{definition}{Definition}{Definitions}
\crefname{remark}{Remark}{Remarks}
\crefname{section}{Section}{Sections}
\crefname{figure}{Figure}{Figures}
\crefname{theoremx}{Theorem}{Theorems}

\newcommand{\Ric}{\operatorname{Ric}}

\DeclareMathOperator{\Vol}{Vol}

\newcommand{\om}{\omega_n}
\newcommand{\snk}{s_{n,k}}
\newcommand{\grd}{g_{\mathrm{rd}}}
\newcommand{\RCD}{\mathrm{RCD}}

\title{A note on $\sigma_k$ maximal conformal volume}
\author{Mijia Lai}
\address{School of Mathematical Sciences, Shanghai Jiao Tong University, Shanghai 200240, China}
\email{laimijia@sjtu.edu.cn}

\begin{document}

\maketitle

\begin{abstract}
We establish a sharp spherical upper bound for the
$\sigma_k$-maximal conformal volume $\Lambda_k$ introduced
by Gursky and Viaclovsky. Let $(M^n,[g])$ be a smooth,
closed, connected conformal manifold with $n\geq 3$ and
$n/2<k\leq n$, and suppose that $[g]$ contains a
$k$-admissible metric, we prove that
$
    \Lambda_k(M,[g])
    \leq \operatorname{Vol}(\mathbb S^n,g_{\mathrm{rd}}),
$
with equality if and only if $(M,[g])$ is conformally
equivalent to the standard sphere. Combined with the
existence and compactness theorem of Gursky and Viaclovsky
below the spherical threshold, this completes their
approach to the $\sigma_k$-Yamabe problem through maximal
conformal volume in the range $k>n/2$. 
\end{abstract}

\section{Introduction}

Let $(M^n,g)$ be a smooth, closed, connected Riemannian manifold
of dimension $n\ge3$. Its Schouten tensor is
\[
 A_g
 =
 \frac{1}{n-2}
 \left(
 \Ric_g-\frac{R_g}{2(n-1)}g
 \right).
\]
For $1\le k\le n$, let $\sigma_k$ denote the unnormalized
elementary symmetric polynomial of degree $k$, and write
\[
 \Gamma_k^+
 =
 \left\{
 \lambda\in\mathbb R^n:
 \sigma_j(\lambda)>0,\quad 1\le j\le k
 \right\}.
\]
A metric $g$ is called $k$-admissible if
$\lambda(g^{-1}A_g)\in\Gamma_k^+$ everywhere. We use the notation
$
 \sigma_k(g)=\sigma_k(g^{-1}A_g).
$
Here and below, a symmetric endomorphism is identified with
its eigenvalue vector when applying an elementary symmetric
polynomial.

Let $\grd$ be the unit round metric on $\mathbb S^n$, and set
\[
 \om=\Vol_{\grd}(\mathbb S^n),
 \qquad
 \snk=\sigma_k(\grd)=2^{-k}\binom nk.
\]
The normalized $\sigma_k$-Yamabe problem asks for a
$k$-admissible metric $h\in[g_0]$ satisfying
\begin{equation}\label{eq:yamabe}
 \sigma_k(h)=\snk.
\end{equation}

Gursky and Viaclovsky \cite{GV04} introduced the
$k$-maximal volume of a conformal class. We call this invariant
the \emph{$\sigma_k$-maximal conformal volume}:
\begin{equation}\label{eq:definition}
 \Lambda_k(M,[g_0])
 =
 \sup_{\substack{
 h\in[g_0],\ \lambda(h^{-1}A_h)\in\Gamma_k^+\\
 \sigma_k(h)\ge\snk
 }}
 \Vol_h(M).
\end{equation}
Following \cite{GV04}, we set $\Lambda_k(M,[g_0])=+\infty$
when the conformal class contains no $k$-admissible metric.
All statements below concerning $\Lambda_k$ assume that such
a metric exists.

For $k>n/2$, admissibility implies positive Ricci curvature,
with a quantitative lower bound \cite{GVW03}.
Bishop's
comparison theorem therefore gives
$\Lambda_k(M,[g_0])<\infty$.
Newton's inequalities also imply
\[
 \Lambda_j(M,[g_0])\ge\Lambda_k(M,[g_0]),
 \qquad \frac n2\le j\le k.
\]

The role of this invariant in the $\sigma_k$-Yamabe problem
comes from the volume threshold established in \cite{GV04}:
for $k>n/2$, the strict inequality
\[
 \Lambda_k(M,[g_0])<\om
\]
implies existence of a solution of \eqref{eq:yamabe}
and compactness of the normalized solution family.
In dimension three, Gursky and Viaclovsky proved
\[
 \Lambda_k(M^3,[g_0])
 \le
 \frac{\omega_3}{|\pi_1(M)|},
 \qquad k=2,3.
\]
In dimension four, they obtained
\[
 \Lambda_k(M^4,[g_0])\le\omega_4,
 \qquad k=2,3,4,
\]
with equality precisely for the spherical conformal class.
Subsequently, Gursky--Viaclovsky \cite{GV07} established
general existence and compactness results for the prescribed
$\sigma_k$ equation when $k>n/2$ on non-spherical conformal
classes. Trudinger--Wang \cite{TW09} proved a stronger
$C^0$ compactness theorem for normalized admissible conformal
metrics.

In this short note, we prove a sharp spherical upper bound for
$\Lambda_k$ $(\frac{n}{2}<k\leq n)$ and the characterization of its equality case.

\begin{theorem}\label{thm:main}
Let $(M^n,[g_0])$ be a smooth, closed, connected conformal
manifold, where $n\ge3$ and $n/2<k\le n$.
Suppose that $[g_0]$ contains a $k$-admissible metric.
Then $\pi_1(M)$ is finite and
\begin{equation}\label{eq:main-bound}
 \Lambda_k(M,[g_0])
 \le
 \frac{\om}{|\pi_1(M)|}.
\end{equation}
Moreover,
\begin{equation}\label{eq:main-rigidity}
 \Lambda_k(M,[g_0])=\om
 \quad\Longleftrightarrow\quad
 (M,[g_0])
 \text{ is conformally equivalent to }
 (\mathbb S^n,[\grd]).
\end{equation}

Every metric in the defining set of \eqref{eq:definition}
satisfies the volume bound in \eqref{eq:main-bound}.
Equality for an individual metric holds if and only if it
is a spherical space form of sectional curvature one.
\end{theorem}

Our proof combines a sharp algebraic inequality with the
determinant volume comparison theorem of Kwong
\cite[Theorem~10]{Kwong}:
 for a closed connected
$n$-manifold with $\Ric_g>0$,
\begin{equation}\label{eq:Kwong}
 \frac{\Vol_g(M)}{\om}
 \le
 \frac{1}{\Vol_g(M)}
 \int_M
 \det\nolimits_g
 \left(\frac{\Ric_g}{n-1}\right)^{-1/2}
 \,dv_g.
\end{equation}
The determinant is taken after raising one index with $g$.

To analyze the equality case, we study the conformal limit of a maximizing sequence of metrics, then compare corresponding
Yamabe constants.
In particular, the proof does not
require the existence of a smooth metric attaining $\Lambda_k(M,[g_0])$.

Combined with
the existence theorem of Gursky and Viaclovsky
\cite[Theorem~1.1]{GV04}, we have the following direct consequence for the $\sigma_k$-Yamabe equation.

\begin{corollary}\label{cor:yamabe}
Let $(M^n,[g_0])$ be a smooth, closed, connected conformal
manifold with $n\geq 3$ and $n/2<k\leq n$. Suppose that
$[g_0]$ contains a $k$-admissible metric. Then there exists
a smooth $k$-admissible metric $h\in[g_0]$ satisfying
\[
    \sigma_k(h^{-1}A_h)=2^{-k}\binom{n}{k}.
\]
Moreover, if $(M,[g_0])$ is not conformally equivalent to
the standard sphere, the set of such solutions is compact
in the $C^m$ topology for every nonnegative integer $m$.
\end{corollary}

\section{Proof of the main theorem}

We first establish the algebraic comparison.
We use the conventions $\sigma_0=1$ and $\sigma_j=0$ if $j$
exceeds the number of variables.
Deleting $r$ coordinates of a vector in $\Gamma_k^+$
leaves a vector in $\Gamma_{k-r}^+$ whenever $r\le k$.

\begin{lemma}\label{lem:algebra}
Let $n\ge3$ and $n/2<k\le n$.
For $\lambda\in\Gamma_k^+$, set
\[
 S=\sum_{i=1}^n\lambda_i,
 \qquad
 \mu_i=(n-2)\lambda_i+S.
\]
Then $\mu_i>0$ and
\begin{equation}\label{eq:algebra}
 \left(\prod_{i=1}^n\mu_i\right)^{1/n}
 \ge
 2(n-1)
 \left(
 \frac{\sigma_k(\lambda)}{\binom nk}
 \right)^{1/k}.
\end{equation}
Equality holds if and only if all $\lambda_i$ are equal.
\end{lemma}

\begin{proof}
The cone estimate of Guan--Viaclovsky--Wang \cite{GVW03},
expressed in Schouten eigenvalues, gives
\begin{equation}\label{eq:cone}
 \lambda_i\ge-\delta S,
 \qquad
 \delta=\frac{n-k}{n(k-1)}.
\end{equation}
Consequently,
\begin{equation}\label{eq:alpha}
 \mu_i\ge\alpha S,
 \qquad
 \alpha=
 \frac{(n-1)(2k-n)}{n(k-1)}>0.
\end{equation}
These nonstrict inequalities extend to
$\overline{\Gamma_k^+}$.

By homogeneity, it suffices to consider
\[
 \mathcal K
 =
 \left\{
 \lambda\in\overline{\Gamma_k^+}:
 \prod_{i=1}^n\mu_i=1
 \right\}.
\]
On this set all $\mu_i$ are positive. Since
$
 \sum_i\mu_i=2(n-1)S,
$
the arithmetic--geometric mean inequality and
\eqref{eq:alpha} give
\[
 \frac{n}{2(n-1)}
 \le S\le\frac1\alpha.
\]
Thus the $\mu_i$ are bounded above and bounded away from
zero. The linear map $\lambda\mapsto\mu$ is invertible, so
$\mathcal K$ is compact.

The diagonal vector
$\lambda_i=1/[2(n-1)]$ belongs to $\mathcal K$ and has
positive $\sigma_k$. Therefore a maximum of $\sigma_k$
on $\mathcal K$ lies in $\Gamma_k^+$.
Indeed, Newton's inequalities imply that $\sigma_k=0$
on the boundary of this cone.

Put $m=n-2$ and regard
\[
 F(\mu)=\sigma_k(\lambda(\mu)),
 \qquad
 \lambda_i
 =
 \frac1m
 \left(
 \mu_i-\frac{\sum_j\mu_j}{2(n-1)}
 \right).
\]
Writing
\[
 P=\sigma_{k-1}(\lambda),
 \qquad
 \beta=\frac{n-k+1}{2(n-1)},
\]
we have
\begin{equation}\label{eq:derivative}
 F_i
 =
 \frac1m
 \left(
 \sigma_{k-1}(\lambda|i)-\beta P
 \right).
\end{equation}
Here $\lambda|i$ denotes deletion of the $i$th coordinate.
At an interior maximum under the constraint
$\sum_i\log\mu_i=0$, the Lagrange multiplier equations give
\begin{equation}\label{eq:stationary}
 \mu_iF_i=\mu_jF_j
 \qquad\text{for all }i,j.
\end{equation}

Choose $j$ such that $y=\lambda_j=\max_i\lambda_i>0$,
and let $x=\lambda_i$ be any other coordinate.
By \eqref{eq:cone}, $S\le ny$, and $2k\ge n+1$,
\[
 x+y
 \ge
 \left(
 1-\frac{n-k}{k-1}
 \right)y
 \ge0.
\]
Let $z$ be the remaining $m$ coordinates, and set
\[
 T=\sigma_1(z),\qquad
 u=\sigma_{k-1}(z),\qquad
 v=\sigma_{k-2}(z)>0,\qquad
 c=m\beta.
\]
A direct calculation from \eqref{eq:derivative} yields
\begin{equation}\label{eq:pair-identity}
 \mu_iF_i-\mu_jF_j
 =
 \frac{x-y}{m}B,
 \qquad
 B=mu-Sv-cP.
\end{equation}

We claim that $B<0$.
If $u\le0$, this is immediate.
If $u>0$, then $z\in\Gamma_{k-1}^+$ and $T>0$.
Newton's inequalities give
\[
 mu\le\frac{n-k}{k-1}Tv.
\]
Since $(n-k)/(k-1)\le1$ and
$S=T+x+y\ge T$, we obtain
\[
 B
 \le
 \frac{n-k}{k-1}Tv-Sv-cP
 \le-cP<0.
\]
For $k=2$, the same argument uses $v=1$; when $u>0$,
the remaining vector has positive sum.

Equations \eqref{eq:stationary} and
\eqref{eq:pair-identity} now force $x=y$.
Since $i$ was arbitrary, every maximum is diagonal.
The product normalization then gives
\[
 \mu_i=1,\qquad
 \lambda_i=\frac1{2(n-1)}.
\]
Hence
\[
 \sigma_k(\lambda)
 \le
 \binom nk[2(n-1)]^{-k}
 \qquad\text{on }\mathcal K.
\]
Rescaling proves \eqref{eq:algebra}.
The argument also proves the asserted equality
characterization.
\end{proof}

Lemma~\ref{lem:algebra}
yields an integral volume inequality for arbitrary
$k$-admissible metrics.

\begin{proposition}\label{pro:integral-volume}
Let $(M^n,g)$ be a smooth, closed, connected Riemannian
manifold with $n\geq 3$ and $n/2<k\leq n$. Suppose that
$g$ is $k$-admissible, and write $d=|\pi_1(M)|$. Then
\[
    \operatorname{Vol}_g(M)^2
    \leq
    \frac{\omega_n}{d}
    \int_M
    \left(
        \frac{2^{-k}\binom{n}{k}}
             {\sigma_k(g^{-1}A_g)}
    \right)^{\frac{n}{2k}}
    \,dv_g.
\]
Equality holds if and only if $(M,g)$ is a spherical
space form of positive constant sectional curvature.
\end{proposition}

\begin{proof}[Proof of Proposition~\ref{pro:integral-volume}]
Set
\[
    V=\operatorname{Vol}_g(M),
    \qquad
    s_{n,k}=2^{-k}\binom{n}{k},
    \qquad
    D_g=\det\nolimits_g
    \left(\frac{\operatorname{Ric}_g}{n-1}\right).
\]
Since $g$ is $k$-admissible and $k>n/2$,
Lemma~\ref{lem:algebra} implies that
$\operatorname{Ric}_g>0$ and
\[
    D_g
    \geq
    \left(
        \frac{\sigma_k(g^{-1}A_g)}{s_{n,k}}
    \right)^{n/k}.
\]
Consequently,
\begin{equation}\label{eq:integral-pointwise}
    D_g^{-1/2}
    \leq
    \left(
        \frac{s_{n,k}}{\sigma_k(g^{-1}A_g)}
    \right)^{n/(2k)}.
\end{equation}

By compactness, the Ricci curvature has a uniform positive
lower bound. The Bonnet--Myers theorem therefore implies
that the universal covering
\[
    \pi:(\widetilde M,\widetilde g)\longrightarrow(M,g),
    \qquad \widetilde g=\pi^*g,
\]
is compact and has finite degree $d=|\pi_1(M)|$.
In particular,
\[
    \operatorname{Vol}_{\widetilde g}(\widetilde M)=dV,
    \qquad
    D_{\widetilde g}=D_g\circ\pi.
\]
Applying the volume comparison theorem
\cite[Theorem~10]{Kwong} to $(\widetilde M,\widetilde g)$
and then using \eqref{eq:integral-pointwise}, we obtain
\begin{align*}
    \frac{dV}{\omega_n}
    &\leq
    \frac{1}{dV}
    \int_{\widetilde M}
        D_{\widetilde g}^{-1/2}\,dv_{\widetilde g}\\
    &=
    \frac{1}{V}
    \int_M D_g^{-1/2}\,dv_g\\
    &\leq
    \frac{1}{V}
    \int_M
    \left(
        \frac{s_{n,k}}{\sigma_k(g^{-1}A_g)}
    \right)^{n/(2k)}
    \,dv_g.
\end{align*}
Multiplying by $\omega_n V/d$ proves the asserted inequality.

Suppose now that equality holds. Then equality holds in
the last integral inequality above. Since the difference
between its two integrands is continuous and nonnegative,
equality holds pointwise in \eqref{eq:integral-pointwise}.
The equality case of Lemma~\ref{lem:algebra} gives
$
    A_g=a\,g
$
for a smooth positive function $a$. Hence
$\operatorname{Ric}_g=2(n-1)a\,g$, and the contracted
Bianchi identity implies that $a$ is constant.
Writing $\kappa=2a>0$, we have
\[
    \operatorname{Ric}_g=(n-1)\kappa g,
    \qquad
    \sigma_k(g^{-1}A_g)=s_{n,k}\kappa^k.
\]
Equality in the asserted integral inequality now yields
$
    V=\frac{\omega_n}{d}\kappa^{-n/2}.
$
Thus
$
    \operatorname{Vol}_{\widetilde g}(\widetilde M)
    =\omega_n\kappa^{-n/2}.
$
The rigidity statement in Bishop's volume comparison
theorem shows that $(\widetilde M,\widetilde g)$ is
isometric to the round sphere of sectional curvature
$\kappa$. Therefore $(M,g)$ is a spherical space form.

Conversely, if $(M,g)$ is a spherical space form of
sectional curvature $\kappa>0$, then
\[
    \sigma_k(g^{-1}A_g)=s_{n,k}\kappa^k,
    \qquad
    V=\frac{\omega_n}{d}\kappa^{-n/2}.
\]
Substitution into the integral inequality gives equality.
\end{proof}
We next record the analytic fact needed to treat equality
of the conformal supremum.
Write
\[
 a_n=\frac{4(n-1)}{n-2},
 \qquad
 p=\frac{2n}{n-2},
\]
and normalize the Yamabe constant by
\[
 Y(M,[g_0])
 =
 \inf_{\psi\in C^\infty(M)\setminus\{0\}}
 \frac{
 \displaystyle
 \int_M
 \left(a_n|\nabla\psi|_{g_0}^2+R_{g_0}\psi^2\right)\,dv_{g_0}
 }{
 \|\psi\|_{L^p(g_0)}^2
 }.
\]

\begin{lemma}\label{lem:limit}
Let $g_j=e^{-2u_j}g_0$ be smooth metrics on a fixed closed
$n$-manifold, $n\ge3$, and suppose that $u_j\to u$ uniformly.
Assume
$
 R_{g_j}\ge n(n-1)
$
and
\[
 b_j
 :=
 \left(
 n-1-\lambda_{\min}(g_j^{-1}\Ric_{g_j})
 \right)_+
 \longrightarrow0
 \quad\text{in }L^1(g_0).
\]
Set $g_\infty=e^{-2u}g_0$ and
$V_\infty=\Vol_{g_\infty}(M)$. Then
\begin{equation}\label{eq:yamabe-limit}
 Y(M,[g_0])
 \ge n(n-1)V_\infty^{2/n}.
\end{equation}
\end{lemma}

\begin{proof}
All derivatives in the following conformal formulas are
taken with respect to $g_0$, and we write $dv_0=dv_{g_0}$.
Scalar positivity implies
\[
 \Delta_0u_j
 \ge
 \frac{n-2}{2}|\nabla_0u_j|^2
 -\frac{R_{g_0}}{2(n-1)}
 \ge-C.
\]
Since $M$ is closed and $\int_M\Delta_0u_j\,dv_0=0$,
we have
\[
 \|\Delta_0u_j\|_{L^1(g_0)}\le C_0.
\]
Integration by parts therefore gives
\[
 \int_M|\nabla_0(u_j-u_\ell)|^2\,dv_0
 \le
 2C_0\|u_j-u_\ell\|_\infty.
\]
Thus
\begin{equation}\label{eq:strong-convergence}
 u_j\longrightarrow u
 \quad\text{strongly in }W^{1,2}(g_0).
\end{equation}

The conformal Ricci formula is
\begin{align}
 \Ric_{g_j}
 &=
 \Ric_{g_0}
 +(n-2)\nabla_0^2u_j
 +(\Delta_0u_j)g_0 \notag\\ \notag
 &\quad
 +(n-2)
 \left(
 du_j\otimes du_j-|\nabla_0u_j|^2g_0
 \right).
\end{align}
The linear second derivative terms converge in
distributions, while the quadratic terms converge in
$L^1$ by \eqref{eq:strong-convergence}.
Hence $\Ric_{g_j}$ converges distributionally to the
Ricci tensor of
$g_\infty\in C^0\cap W^{1,2}$.

Moreover,
$
 \Ric_{g_j}\ge(n-1-b_j)g_j.
$
Testing against a fixed smooth vector field and a
nonnegative smooth function, uniform convergence of the
metrics and $b_j\to0$ in $L^1(g_0)$ imply
\begin{equation}\label{eq:distributional-Ricci}
 \Ric_{g_\infty}\ge(n-1)g_\infty
 \quad\text{in the sense of distributions}.
\end{equation}
The metric $g_\infty$ is continuous and uniformly equivalent
to $g_0$, with Christoffel symbols in $L^2$.
Completeness and the required volume growth bound follow
from compactness.
Theorem~7.2 and Remark~7.3 of Mondino--Ryborz \cite{MR}
therefore identify \eqref{eq:distributional-Ricci} with
the corresponding $\RCD(n-1,n)$ curvature bound.

The sharp Sobolev inequality in this setting,
with probability measure $dv_{g_\infty}/V_\infty$,
gives
\begin{equation}\label{eq:sharp-Sobolev}
 \begin{split}
 &a_n\int_M|\nabla f|_{g_\infty}^2\,dv_{g_\infty}
 +n(n-1)\int_M f^2\,dv_{g_\infty}\\
 &\qquad\ge
 n(n-1)V_\infty^{2/n}
 \|f\|_{L^p(g_\infty)}^2,
 \qquad f\in W^{1,2}(g_\infty).
 \end{split}
\end{equation}
See \cite[Theorem~1.8]{NV}.
The Sobolev energy here agrees with the energy computed
from the continuous metric tensor.

Fix $\psi\in C^\infty(M)$ and put
\[
 w_j=e^{-(n-2)u_j/2},
 \qquad
 f_j=\frac{\psi}{w_j},
 \qquad
 w=e^{-(n-2)u/2},
 \qquad
 f=\frac{\psi}{w}.
\]
The functions $w_j$ are uniformly positive and bounded.
Equation \eqref{eq:strong-convergence} implies
$f_j\to f$ strongly in $W^{1,2}(g_0)$.
Conformal covariance of the conformal Laplacian gives
\begin{align*}
 &\int_M
 \left(a_n|\nabla\psi|_{g_0}^2+R_{g_0}\psi^2\right)\,dv_0\\
 &\quad=
 a_n\int_M|\nabla f_j|_{g_j}^2\,dv_{g_j}
 +\int_M R_{g_j}f_j^2\,dv_{g_j}\\
 &\quad\ge
 a_n\int_M|\nabla f_j|_{g_j}^2\,dv_{g_j}
 +n(n-1)\int_M f_j^2\,dv_{g_j}.
\end{align*}
The two terms in the last line converge to their
$g_\infty$ counterparts.
Applying \eqref{eq:sharp-Sobolev} and observing that
\[
 \int_M|f|^p\,dv_{g_\infty}
 =
 \int_M|\psi|^p\,dv_0,
\]
we obtain
\[
 \int_M
 \left(a_n|\nabla\psi|_{g_0}^2+R_{g_0}\psi^2\right)\,dv_0
 \ge
 n(n-1)V_\infty^{2/n}
 \|\psi\|_{L^p(g_0)}^2.
\]
Taking the infimum over $\psi$ proves
\eqref{eq:yamabe-limit}.
\end{proof}

\begin{proof}[Proof of Theorem~\ref{thm:main}]
Let $h$ belong to the defining set of
\eqref{eq:definition}.
The Ricci eigenvalues of $h$ are
\[
 \mu_i=(n-2)\lambda_i+\sum_j\lambda_j,
\]
where $\lambda_i$ are its Schouten eigenvalues.
Lemma~\ref{lem:algebra} implies
\begin{equation}\label{eq:determinant-lower}
 D_h
 :=
 \det\nolimits_h
 \left(\frac{\Ric_h}{n-1}\right)
 \ge
 \left(
 \frac{\sigma_k(h)}{\snk}
 \right)^{n/k}
 \ge1.
\end{equation}
Using \eqref{eq:Kwong}, we conclude that
$ \Vol_h(M)\le\om.$

If $\Vol_h(M)=\om$, then the nonnegative smooth function
$1-D_h^{-1/2}$ has zero integral.
Thus $D_h\equiv1$.
Equality holds throughout \eqref{eq:determinant-lower},
so Lemma~\ref{lem:algebra} gives
\[
 A_h=\frac12h,
 \qquad
 \Ric_h=(n-1)h.
\]
Bishop's volume rigidity then shows that $(M,h)$ is
the unit round sphere.

The cone estimate and Newton's inequalities also give
\begin{equation}\label{eq:uniform-curvature}
 R_h\ge n(n-1),
 \qquad
 \Ric_h\ge c_{n,k}h,
 \qquad
 c_{n,k}
 =
 \frac{(n-1)(2k-n)}{2(k-1)}>0.
\end{equation}
Myers' theorem implies that the universal cover is compact
and $d:=|\pi_1(M)|<\infty$.
Applying the preceding volume estimate to the lifted
metric yields
$
 d\,\Vol_h(M)\le\om.
$
Taking the supremum proves \eqref{eq:main-bound}.
Equality for an individual metric occurs precisely when
its universal cover is the unit round sphere, proving
the metric equality statement.

It remains to prove equality rigidity for $\Lambda_k$.
The round conformal class contains $\grd$, so its maximal
conformal volume equals $\om$.
Conversely, suppose that $[g_0]$ is not the spherical
conformal class and that
$
 \Lambda_k(M,[g_0])=\om.
$
Choose smooth metrics $g_j=e^{-2u_j}g_0$ in the defining
set such that
\[
 V_j:=\Vol_{g_j}(M)\longrightarrow\om.
\]

Theorem~A of Trudinger--Wang \cite{TW09} gives
$C^0$ precompactness of admissible conformal metrics
after fixing the volume.
To apply it, first rescale $g_j$ to a common volume;
the rescaling constants converge because $V_j\to\om$.
Consequently, after passing to a subsequence,
\[
 u_j\longrightarrow u
 \quad\text{uniformly},
 \qquad
 \Vol_{e^{-2u}g_0}(M)=\om.
\]
In particular, $dv_{g_j}$ and $dv_{g_0}$ are uniformly
equivalent.

Set $D_j=D_{g_j}$.
By \eqref{eq:Kwong} and \eqref{eq:determinant-lower},
\[
 0
 \le
 \frac1{V_j}\int_M(1-D_j^{-1/2})\,dv_{g_j}
 \le
 1-\frac{V_j}{\om}
 \longrightarrow0.
\]
It follows that $D_j\to1$ in measure.

We claim that all Ricci eigenvalues converge to $n-1$
in measure.
Indeed, \eqref{eq:uniform-curvature} bounds them below
by $c_{n,k}>0$.
On a bounded sublevel of $D_j$, their product is bounded,
so the spectra lie in a compact set.
Within this set, the only spectrum satisfying
$D=1$ and $\sigma_k\ge\snk$ is
\[
 (n-1,\ldots,n-1),
\]
by Lemma~\ref{lem:algebra}.
The limiting Schouten vector remains in $\Gamma_k^+$
because Newton's inequalities give positive lower bounds
for all $\sigma_\ell$, $1\le\ell\le k$.
This proves the claim by compactness.

Therefore
\[
 b_j
 =
 \left(
 n-1-\lambda_{\min}(g_j^{-1}\Ric_{g_j})
 \right)_+
 \longrightarrow0
 \quad\text{in measure}.
\]
Since $0\le b_j\le n-1$, the convergence holds in
$L^1(g_0)$.
Together with $R_{g_j}\ge n(n-1)$,
Lemma~\ref{lem:limit} gives
\[
 Y(M,[g_0])
 \ge
 n(n-1)\om^{2/n}
 =
 Y(\mathbb S^n,[\grd]).
\]
The Aubin--Schoen theorem asserts that the reverse
inequality always holds, with equality only for the
spherical conformal class; see \cite{LP87}.
This contradicts our assumption on $[g_0]$ and proves
\eqref{eq:main-rigidity}.
\end{proof}


\begin{thebibliography}{99}

\bibitem{GV04}
M.~J.~Gursky and J.~A.~Viaclovsky,
\emph{Volume comparison and the $\sigma_k$-Yamabe problem},
Adv. Math. \textbf{187} (2004), 447--487.


\bibitem{GV07}
M.~J.~Gursky and J.~A.~Viaclovsky,
\emph{Prescribing symmetric functions of the eigenvalues
of the Ricci tensor},
Ann. of Math. (2) \textbf{166} (2007), 475--531.


\bibitem{GVW03}
P.~Guan, J.~Viaclovsky and G.~Wang,
\emph{Some properties of the Schouten tensor and
applications to conformal geometry},
Trans. Amer. Math. Soc. \textbf{355} (2003), 925--933.

\bibitem{Kwong}
K.-K.~Kwong,
\emph{An improved volume bound under Ricci and scalar
curvature lower bounds},
preprint, 2026, version~2, Theorem~10.
\href{https://arxiv.org/abs/2608.19196v2}
{arXiv:2608.19196v2}.

\bibitem{LP87}
J.~M.~Lee and T.~H.~Parker,
\emph{The Yamabe problem},
Bull. Amer. Math. Soc. (N.S.) \textbf{17} (1987), 37--91.


\bibitem{MR}
A.~Mondino and V.~Ryborz,
\emph{On the equivalence of distributional and synthetic
Ricci curvature lower bounds},
J. Funct. Anal. \textbf{289} (2025), article 111035.


\bibitem{NV}
F.~Nobili and I.~Y.~Violo,
\emph{Rigidity and almost rigidity of Sobolev inequalities
on compact spaces with lower Ricci curvature bounds},
Calc. Var. Partial Differential Equations
\textbf{61} (2022), article 180.

\bibitem{TW09}
N.~S.~Trudinger and X.-J.~Wang,
\emph{On Harnack inequalities and singularities of
admissible metrics in the Yamabe problem},
Calc. Var. Partial Differential Equations
\textbf{35} (2009), 317--338.


\end{thebibliography}
\end{document}